\documentclass[11pt,a4paper,reqno,oneside]{amsart}
\usepackage{tabls}
\usepackage{url}
\usepackage[linktocpage = true]{hyperref}
\hypersetup{
	colorlinks = true,
	citecolor = blue,
}
\usepackage{color}
\definecolor{red}{rgb}{1,0,0}
\definecolor{green}{rgb}{0,1,0}
\definecolor{blue}{rgb}{0,0,1}
\definecolor{refkey}{gray}{.625}
\definecolor{labelkey}{gray}{.625}
\usepackage[lite]{amsrefs}
\usepackage[a4paper]{geometry}
\usepackage{amsfonts}
\usepackage{amssymb}
\usepackage{mathrsfs}
\usepackage{amsmath}
\usepackage{amsthm}
\usepackage{tikz}
\usepackage{tikz-cd}
\usepackage{amscd}
\usepackage{enumerate}
\usepackage{paralist}
\usepackage{graphicx}
\usepackage{mathtools}
\usepackage{kantlipsum}
\usepackage{stmaryrd}
\usepackage{extarrows}
\allowdisplaybreaks

\makeatletter
\def\title@font{\normalsize\bfseries}
\let\ltx@maketitle\@maketitle
\def\@maketitle{\bgroup%
	\let\ltx@title\@title%
	\def\@\title{\resizebox{\textwidth}{!}{%
			\mbox{\title@font\ltx@title}%
	}}%
	\ltx@maketitle%
	\egroup}
\makeatother

\theoremstyle{plain}
\newtheorem*{zorn*}{Zorn's lemma}
\newtheorem*{tychonoff*}{Tychonoff's theorem}

\newtheorem{definition}{Definition}[section]
\newtheorem{lemma}[definition]{Lemma}
\newtheorem{Cor}[definition]{Corollary}
\newtheorem{theorem}[definition]{Theorem}
\newtheorem*{theorem*}{Theorem}

\newtheorem{def-prop}[definition]{Definition-Proposition}
\newtheorem{proposition}[definition]{Proposition}
\newtheorem{prop-def}[definition]{Proposition-Definition}
\newtheorem{example}[definition]{Example}

\DeclareMathOperator{\CDO}{CDO}

\newcommand {\emptycomment}[1]{}

\newcommand{\CinfM}{C^\infty(M)}

\usepackage{stmaryrd}

\newcommand{\R}{\mathbb{R}}

\begin{document}
	\title{
		The geometric constraints on Filippov  algebroids
	}
	\author{Yanhui Bi}
	\address{Center for Mathematical Sciences, College of Mathematics and Information Science, Nanchang Hangkong University} 
	\email{\href{mailto:biyanhui@nchu.edu.cn}{biyanhui@nchu.edu.cn}}

	\author{Zhixiong Chen}
	\address{College of Mathematics and Information Science, Nanchang Hangkong University} 
	\email{\href{mailto:chenzhixiong0908@foxmail.com}{chenzhixiong0908@foxmail.com}}
	
	\author{Zhuo Chen}
	\address{Department of Mathematical Sciences, Tsinghua University} 
	\email{\href{mailto:chenzhuo@tsinghua.edu.cn}{chenzhuo@tsinghua.edu.cn}}
	
	
	\author{Maosong Xiang}
	\address{School of Mathematics and Statistics, Huazhong University of Science and Technology}
	\email{\href{mailto: msxiang@hust.edu.cn}{msxiang@hust.edu.cn}}
	\thanks{Supported by the National Natural Science Foundation of China (NSFC)  grants 11961049(Bi),  11901221(Xiang), and 12071241(Zhuo Chen), and by the Key Project of Jiangxi Natural Science Foundation grant  20232ACB201004(Bi, Xiang and Zhixiong Chen).}
	\date{}
	\begin{abstract}
		Filippov $n$-algebroids are introduced by Grabowski and Marmo as 	a natural generalization of Lie algebroids. On this note, we characterized Filippov $n$-algebroid structures by considering certain multi-input connections, which we called Filippov connections, on the underlying vector bundle. Through this approach, we could express the $n$-ary bracket of any Filippov $n$-algebroid using a torsion-free type formula. Additionally, we transformed the generalized Jacobi identity of the Filippov $n$-algebroid into the Bianchi-Filippov identity. Furthermore, in the case of rank $n$ vector bundles, we provided a characterization of linear Nambu-Poisson structures using Filippov connections.

	\end{abstract}
	\maketitle
	\section*{Introduction}

		Filippov introduced a generalized Jacobi identity for $n$-ary skew-symmetric operation, which acts as a replacement for the classical Jacobi identity in the context of Lie algebras \cite{ref10}. He also proposed the concept of $n$-Lie algebra, also known as Filippov $n$-algebra, with the corresponding generalized Jacobi identity referred to as the Filippov identity. Nambu and Takhtajan extended the concept of Poisson manifold to an $n$-ary generalization called Nambu-Poisson structure in order to study Hamiltonian mechanics more comprehensively \cite{ref16,Tak}. It is worth noting that both the Nambu-Poisson structure and the $n$-Lie algebra share the same generalized Jacobi identity. Grabowski and Marmo introduced the concept of Filippov $n$-algebroids, an $n$-ary generalization of Lie algebroids, in order to determine the relationship between linear Nambu-Poisson structures and Filippov algebras \cite{ref12}. Consequently, it is reasonable to anticipate that many tools used to study Lie algebroids could be enhanced or upgraded to the realm of Filippov algebroids. Therefore, we aim to address the absence of the concepts of connections and curvatures of Filippov algebroids in the literature and provide a primitive analysis of these topics from a geometric point of view.
	
	Recall that a Lie algebroid  is a (real) vector bundle $A\to M$ together with a bundle map $\rho\colon A\to TM$, called anchor, and a Lie bracket $[\cdotp,\cdotp]$ on the section space $\Gamma(A)$ of $A$, satisfying that $\rho\colon \Gamma(A)\to \Gamma(TM)$ is a morphism of Lie algebras and the Leibniz rule
	\[
	[X, fY] =f[X, Y] +(\rho(X)f)Y, \qquad \forall X, Y\in \Gamma(A) \mbox{ and } f\in \CinfM.
	\]
	By an easy smooth analysis, the bracket $[\cdotp,\cdotp]$ can always be reformulated in the form
	\begin{equation}\label{Eqt:XYcommute}
		[X,Y]=\nabla _XY-\nabla _YX,
	\end{equation}
	where $\nabla\colon \Gamma(A)\times \Gamma(A)\to \Gamma(A)$ satisfies the properties
	\[
	\nabla _{fX}Y=f \nabla_XY \mbox{ and } \nabla _{X}(fY)=f \nabla_XY + (\rho(X)f) Y.
	\]
	One calls $\nabla$ a connection on the anchored bundle $(A,\rho)$. (This is indeed a straightforward generalization of connections on vector bundles.)
	
	When the bracket $[-,-]$ of a Lie algebroid $A$ is expressed in the form \eqref{Eqt:XYcommute}, one says that the connection $\nabla$ is torsion free. See~\cite{LSX} for the existence of torsion free connections on Lie algebroids.
	The curvature form $R^{\nabla} \in \Gamma(\wedge^2 A^* \otimes \operatorname{End}(A))$ of such a connection $\nabla$ is defined in the standard manner:
	\[
	R^{\nabla }(X,Y)(Z)=\nabla _X\nabla _YZ-\nabla _Y\nabla _XZ-\nabla _{[X,Y]}Z,
	\]
	for all $X,Y,Z \in \Gamma(A)$. Now, $\rho$ being a morphism of Lie algebras is equivalent to the condition that $R^{\nabla}$ is a tensor in its third argument.  Moreover, the Jacobi identity for $[~,~]$ is transformed into the following Lie-Bianchi identity
	\begin{equation*}
		R^{\nabla}(X,Y)(Z)+R^{\nabla}(Y,Z)(X)+R^{\nabla}(Z,X)(Y)=0.
	\end{equation*}
	Therefore, Lie algebroids can be realized as anchored bundles equipped with  special connections~\cite{PP}. We wish to find an analogous characterization of the $n$-ary bracket of any Filippov algebroid.
	A significant difference between Lie algebroids and Filippov $n$-algebroids (for $n\geqslant 3$) is that the bracket and anchor of the latter are of more arguments (see Definition \ref{def1.2}). So there is not an obvious way to extend Eq \eqref{Eqt:XYcommute}.
	We come up with a solution in Section \ref{Sec:mainpart}. Below is a quick summary:
	\begin{itemize}
		\item First, we define (multi-input) connections compatible with a given (multi-input) $n$-anchor (see Definition \ref{def2.1}). This is a quite straightforward extension of usual connections of Lie algebroids (when $n=2$).
		\item Second, we introduce the curvature form $R^\nabla$ stemming from a connection $\nabla$ (see Eq \eqref{eq10}). We believe that this is a highly nontrivial invention of this note.
		\item Third,  we prove in Theorem \ref{th2.3} that certain good connections, which we call Filippov connections, fully determines Filippov algebroid structures.  This includes two points: (1) The $n$-ary bracket of any Filippov algebroid can be realized in a torsion free manner (see Eq \eqref{eq9}); (2) The generalized Jacobi identity is transformed to a constraint, called the Bianchi-Filippov identity (see Eq \eqref{eq12}) about the associated curvature $R^\nabla$.
	\end{itemize}
	
	We then illustrate a simple method via covariant differential operators to construct Filippov connections in Section \ref{Sec:Dxipair}.
	
	As vector bundles are fiber bundles with linear fibers, particular cases of homogeneity structures,  linear geometrical structures on vector bundles are of particular interest. (see also~\cite{GGR} on weighted structures for various geometric objects on manifolds with general homogeneity structures.)
	We finally show that  there exists a one-to-one correspondence between Filippov $n$-algebroid structures on a vector bundle $A$ of rank $n \geqslant 3$ and linear Nambu-Poisson structures on its dual bundle $A^*$ (see Theorem~\ref{prop3.3}).
	As an interesting application of our result, one is able to construct  linear Nambu-Poisson structures from Filippov connections (Corollary \ref{propo3.3}).

	In short summary, torsion-free connections subject to the Bianchi-Filippov identity are important geometric constraints for Filippov algebroids. It is well known that torsion free connections for Lie algebroids play a crucial role in various mathematical constructions, for example, in the construction of Poincar\'{e}-Birkhoff-Witt isomorphisms and Kapranov dg manifolds for Lie algebroid pairs~\cite{LSX}. Additionally, Bianchi identities are not only significant in Riemannian geometry, but also in Poisson geometry~\cite{BDPR}.
	We believe that our approach to Filippov algebroids  will be beneficial in this context.		
	\section{Preliminaries: Anchored bundles and Filippov algebroids}\label{Sec:pre}
	In this section, we recall the definition of Filippov algebroids from \cite{ref12}. There is an alternative characterization of Filippov algebroids in terms of certain $1$-derivations \cite{ref14}. It is important to note that $n\geqslant 2$ is an integer, although the only interesting situation is when $n\geqslant 3$.	Let us start with a   notion of $n$-anchored bundles.
	\begin{definition}\label{def1.1}
		An $n$-anchored vector bundle over a smooth manifold $M$ is pair $(A,\rho)$, where $A$ is a vector bundle over $M$ and $\rho\colon \wedge^{n-1}A \to TM$ is a vector bundle morphism, called $n$-anchor of $A$.
	\end{definition}
	\begin{definition}\label{def1.2}
		A Filippov $n$-algebroid over a smooth manifold $M$ is an $n$-anchored bundle $(A,\rho)$ over $M$ together with an $\mathbb{R}$-multilinear and skew-symmetric $n$-bracket  on the section space $\Gamma(A)$ of  $A$:
		\[
		[\cdotp,\cdots,\cdotp]:~\underbrace{ \Gamma(A)\times \cdotp \cdotp \cdotp \times  \Gamma(A)}_{n-\mbox{copies}} \to \Gamma(A)
		\]
		satisfying the following compatibility conditions:
		\begin{enumerate}
			\item[(1)]The $n$-anchor $\rho$ intertwines the $n$-bracket and the standard Lie bracket $[\cdotp,\cdotp]_{TM}$ on $\Gamma(TM)$:
			\begin{eqnarray}\label{eq1}
				[\rho(X_1\wedge \cdotp \cdotp \cdotp \wedge X_{n-1}),\rho(Y_1\wedge \cdotp \cdotp \cdotp \wedge Y_{n-1})]_{TM}
				=\sum_{i=1}^{n-1}\rho(Y_1\wedge \cdotp \cdotp \cdotp \wedge [X_1,\cdots,X_{n-1},Y_i]\wedge\cdotp \cdotp \cdotp \wedge Y_{n-1});			
			\end{eqnarray}
			\item[(2)]The $n$-bracket is a derivation with respect to $C^\infty(M)$-multiplications:
			\begin{equation}\label{eq2}
				[X_1,\cdots,X_{n-1},fY]=f[X_1,\cdots,X_{n-1},Y]+\rho(X_1\wedge \cdotp \cdotp \cdotp \wedge X_{n-1})(f)Y;
			\end{equation}
			\item[(3)]The following equation holds, to be called the (generalized) Jacobi identity (or Filippov identity):
			\begin{eqnarray}\label{eq3}
				[X_1,\cdots,X_{n-1},[Y_1,\cdots,Y_{n}]]
				=\sum_{i=1}^{n}[Y_1,\cdots,Y_{i-1},[X_1,\cdots,X_{n-1},Y_i],Y_{i+1},\cdots,Y_n],
			\end{eqnarray}
		\end{enumerate}	for all $X_{i},Y_{i}\in \Gamma(A)$ and $f\in C^\infty(M)$.
	\end{definition}
	
	Note that any Lie algebroid is a Filippov 2-algebroid. A Filippov $n$-algebra is a Filippov $n$-algebroid over the one-point base manifold.
	In fact, analogous to the Lie algebroid case (i.e. $n=2$ case), the condition $(1)$ in the above definition follows from the conditions $(2)$ and $(3)$.
	
	The following examples (due to~\cite{ref12}) illustrate two Filippov $n$-algebroid structures on the trivial tangent bundle $T\mathbb{R}^m$ of $\mathbb{R}^m$ for $m  \geqslant n\geqslant 2$.
	\begin{example}\label{example1.3}
		Consider the trivial n-anchored bundle $(T\mathbb{R}^m, \rho = 0)$. For each Filippov $n$-algebra structure on $\R^m$ with structure constants $\{c^j_{i_1,\cdots,i_n}\}$ and each  smooth function $g\in C^\infty(\mathbb{R}^m)$, we have a Filippov $n$-algebroid   $(T\mathbb{R}^m,0)$ whose bracket is defined by
		\begin{equation*}
			\left[f_1\frac{\partial}{\partial x_{i_1}},\cdots,f_n\frac{\partial}{\partial x_{i_n}}\right]=gf_1\cdots f_n \sum_{j=1}^{m}c^j_{i_1,\cdots,i_n}\frac{\partial}{\partial x_{j}}.
		\end{equation*}	
	\end{example}
	\begin{example}\label{example1.4}
		Equip $T\mathbb{R}^m$ with the $n$-anchor map $\rho$ defined by the tensor field
		\[
		dx_1 \wedge \cdotp \cdotp \cdotp \wedge dx_{n-1} \otimes \frac{\partial}{\partial x_1},
		\]
		where $x_1,\cdots,x_{n-1},x_n,\cdots,x_m$ are coordinates of $\mathbb{R}^m$. 		
		Then, the $n$-anchored bundle $(T\mathbb{R}^m ,\rho)$ together with the trivial $n$-bracket on generators $\frac{\partial}{\partial x_{i}}$ produces a (nontrivial) Filippov $n$-algebroid over $\mathbb{R}^m$.			
	\end{example}	
	
	We emphasize a crucial but often overlooked point in the literature: the presence of a Filippov $n$-bracket on an $n$-anchored bundle $(A,\rho)$ imposes a constraint on the rank of $\rho$ for every integer $n\geqslant 3$.
	
	\begin{proposition}\label{propo1.5}
		Let $(A,[\cdotp,\cdotp \cdotp \cdotp,\cdotp],\rho)$ be a Filippov $n$-algebroid for $n\geqslant 3$. Then the rank of the image of $\rho$ as a distribution on $M$ can not exceed   $1$, i.e., $\operatorname{rank}(\rho(\wedge^{n-1}A))\leqslant 1$.
	\end{proposition}
	\begin{proof}\renewcommand{\qedsymbol}{}
		Suppose that the image of $\rho$ at $p\in M$ is not trivial. So we can find an  open neighborhood $U$ of $p$ and some $Y_1\wedge \cdotp \cdotp \cdotp \wedge Y_{n-1}\in \Gamma(\wedge^{n-1}A)|_U$ such that  $\rho(Y_1\wedge \cdotp \cdotp \cdotp \wedge Y_{n-1}) $ is nowhere vanishing on $U$. 	The desired statement amounts to show that, if  $\rho(X_1\wedge \cdotp \cdotp \cdotp \wedge X_{n-1}) $ is also nowhere vanishing on   $U $, then   there exists some $c \in C^\infty(U)$ such that
		\[\rho(X_1\wedge \cdots \wedge X_{n-1})=c \rho(Y_1\wedge \cdots \wedge Y_{n-1}).	\]
		
		In fact, by the definition of Filippov $n$-algebroids, we obtain
		\begin{align*}
			&\quad [\rho(fX_1\wedge \cdots \wedge X_{n-1}),\rho(Y_1\wedge \cdots \wedge Y_{n-1})] \qquad \text{by Eq.~\eqref{eq1}}\\
			&= \sum_{i=1}^{n-1}\rho(Y_1\wedge \cdots \wedge Y_{i-1} \wedge[fX_1,X_2,\cdots,X_{n-1},Y_i] \wedge Y_{i+1}\wedge \cdots \wedge Y_{n-1}) \quad \text{by Eq.~\eqref{eq2}} \\
			&= f\sum_{i=1}^{n-1}\rho(Y_1\wedge \cdots \wedge Y_{i-1}\wedge [X_1,X_2,\cdots,X_{n-1},Y_i] \wedge Y_{i+1} \wedge \cdots \wedge Y_{n-1})			\\
			&\quad +\sum_{i=1}^{n-1}(-1)^{n-1}\rho(X_2\wedge\cdots\wedge X_{n-1}\wedge Y_i)(f)\rho(Y_1\wedge \cdots \wedge Y_{i-1}\wedge X_1\wedge Y_{i+1}\wedge \cdots \wedge Y_{n-1})  \quad \text{by Eq.~\eqref{eq1}} \\
			&= f[\rho(X_1\wedge \cdots \wedge X_{n-1}),\rho(Y_1\wedge \cdots \wedge Y_{n-1})]  \\
			&\quad +\sum_{i=1}^{n-1}(-1)^{n-1}\rho(X_2\wedge\cdots\wedge X_{n-1}\wedge Y_i)(f)\rho(Y_1\wedge \cdots \wedge Y_{i-1}\wedge X_1\wedge Y_{i+1}\wedge \cdots \wedge Y_{n-1}).
		\end{align*}
		Moreover, since $\rho$ is a morphism of vector bundles, we have
		\begin{eqnarray*}
			&&[\rho(fX_1\wedge \cdotp \cdotp \cdotp \wedge X_{n-1}),\rho(Y_1\wedge \cdotp \cdotp \cdotp \wedge Y_{n-1})] \\
			&=& [f\rho(X_1\wedge \cdotp \cdotp \cdotp \wedge X_{n-1}),\rho(Y_1\wedge \cdotp \cdotp \cdotp \wedge Y_{n-1})]  \\
			&=&f[\rho(X_1\wedge \cdotp \cdotp \cdotp \wedge X_{n-1}),\rho(Y_1\wedge \cdotp \cdotp \cdotp \wedge Y_{n-1})]-\rho(Y_1\wedge \cdotp \cdotp \cdotp \wedge Y_{n-1})(f)\rho(X_1\wedge \cdotp \cdotp \cdotp \wedge X_{n-1}).	
		\end{eqnarray*}
		Setting $Y_1 = X_1$ in the above two equations,  we obtain
		\begin{equation}\label{eq6}
			\rho(X_ 1 \wedge \cdotp \cdotp \cdotp \wedge X_{n-1})(f)\rho(X_1\wedge Y_2 \wedge \cdotp \cdotp \cdotp \wedge Y_{n-1})= -\rho(X_1\wedge Y_2 \wedge \cdotp \cdotp \cdotp \wedge Y_{n-1})(f)\rho(X_1\wedge \cdotp \cdotp \cdotp \wedge X_{n-1}).
		\end{equation}
		
		Using Eq.~\eqref{eq6}, we have
		\begin{align*}
			\rho(X_1\wedge \cdotp \cdotp \cdotp \wedge X_{n-1}) &= g_1\rho(X_1 \wedge Y_2\wedge \cdotp \cdotp \cdotp \wedge Y_{n-1}) = -g_1 \rho(Y_2 \wedge X_1 \wedge \cdots \wedge Y_{n-1}) \\
			&= -g_1g_2 \rho(Y_2 \wedge Y_1 \wedge \cdotp \cdotp \cdotp \wedge Y_{n-1}) = g_1g_2 \rho(Y_1\wedge  \cdotp \cdotp \cdotp \wedge Y_{n-1}),
		\end{align*}
		for some $g_1,g_2 \in C^\infty(U)$.
		\begin{enumerate}
			\item[(1)] If $\rho(X_1\wedge Y_2\wedge \cdotp \cdotp \cdotp \wedge Y_{n-1}) $ is nowhere vanishing on $U$, then the vector fields $\rho(X_1\wedge \cdotp \cdotp \cdotp \wedge X_{n-1})$ and $\rho(Y_1\wedge \cdotp \cdotp \cdotp \wedge Y_{n-1})$ must be $C^\infty(U)$-linearly dependent.
			\item[(2)] If $\rho(X_1\wedge Y_2\wedge \cdotp \cdotp \cdotp \wedge Y_{n-1})|_p = 0$, then we let $\tilde{X}_1=X_1+Y_1$ and consider $\rho(\tilde{X}_1\wedge Y_2\wedge \cdotp \cdotp \cdotp \wedge Y_{n-1}) $, which is nowhere vanishing on $U$. By arguments in (1) as above, $\rho(\tilde{X}_1\wedge X_2\wedge \cdotp \cdotp \cdotp \wedge X_{n-1})$ and $\rho(Y_1\wedge \cdotp \cdotp \cdotp \wedge Y_{n-1})$ are $C^\infty(U)$-linearly dependent, and we obtain the desired statement as well.
		\end{enumerate}
	\end{proof}
	
	\section{The geometric constraints of Filippov algebroids}\label{Sec:mainpart}
	\subsection{The main theorem}
	In this section, we characterize Filippov algebroids via connections on the underlying anchored bundles.
	\begin{definition}\label{def2.1}
		A connection on an $n$-anchored bundle $(A,\rho)$ is a bilinear map $\nabla: \Gamma(\wedge^{n-1}A)\times \Gamma(A)\to \Gamma(A)$ satisfying two conditions:
		\begin{eqnarray*}		
			\nabla_{fX_1\wedge \cdots \wedge X_{n-1}}X_n&=&f\nabla_{X_1\wedge \cdots \wedge X_{n-1}}X_n, \\\mbox{ and }~
			\nabla_{X_1\wedge \cdots \wedge X_{n-1}}(fX_n)&=&f\nabla_{X_1\wedge \cdots \wedge X_{n-1}}X_n+\rho(X_1\wedge \cdots \wedge X_{n-1})(f)X_n
		\end{eqnarray*}
		for all $X_1, \cdots, X_{n}\in \Gamma(A)$ and $f \in C^\infty(M)$.
	\end{definition}
	To see the existence of such a connection, one takes  a $TM$-connection on $A$, say $\nabla^{TM}$, and then define $\nabla$ on the $n$-anchored bundle $(A,\rho)$ as the  pullback of $\nabla^{TM}$:
	\[
	\nabla_{X_1\wedge \cdots \wedge X_{n-1}}X_n:=\nabla^{TM}_{\rho(X_1\wedge \cdots \wedge X_{n-1})}X_n.
	\]
	
	The key point of this note is that any connection $\nabla$ on $(A, \rho)$ induces a skew-symmetric $n$-bracket on $\Gamma(A)$ defined by
	\begin{eqnarray}\label{eq9}
		[X_1,\cdots,X_n]^{\nabla}\notag
		&:=& \sum_{i=1}^{n}(-1)^{n+i}\nabla_{X_{1}\wedge \cdotp \cdotp \cdotp \widehat{X_{i}}\cdotp \cdotp \cdotp \wedge X_{n} }X_i\notag	\\
		&=& \sum_{i=1}^{n}(-1)^{(n-1)i}\nabla_{X_{i+1}\wedge \cdotp \cdotp \cdotp \wedge X_{n}\wedge X_1\wedge \cdotp \cdotp \cdotp \wedge X_{i-1}}X_i.
	\end{eqnarray}	
	
	For computational convenience,  we denote the covariant derivative on $\Gamma(A)$ along $X_1, \cdots, X_{n-1} \in \Gamma(A)$ by
	\[
	X^{\nabla}_{1\cdots n-1}:=[X_1,\cdots,X_{n-1},-]^{\nabla} \colon \Gamma(A) \to \Gamma(A).
	\]
	It extends to all sections in $\wedge^\bullet A$ by
	\[
	{X_{1\cdots n-1}^\nabla}(Y_1\wedge\cdotp \cdotp \cdotp \wedge Y_{m}):=\sum_{i=1}^{m}Y_1\wedge\cdotp\cdotp\cdotp \wedge Y_{i-1} \wedge X^{\nabla}_{1\cdots n-1}(  Y_i)\wedge Y_{i+1}\wedge \cdotp\cdotp\cdotp\wedge Y_{m}.
	\]
	We then introduce the curvature form of $\nabla$, an operation
	\[
	R^\nabla(-\cdots-,-)(-):~\underbrace{\Gamma( A)\times \cdots \times\Gamma(A)}_{(n-1)-\mbox{copies}}\times \Gamma(\wedge^{n-1}A) \times \Gamma(A) \to \Gamma(A),
	\]
	defined by
	\begin{eqnarray}\label{eq10}
		&&R^\nabla(X_1,\cdots,X_{n-1},Y_1\wedge\cdotp \cdotp \cdotp \wedge Y_{n-1})(Z)\notag
		\\
		&:=&[ {X_{1\cdots n-1}^\nabla},\nabla_{Y_1\wedge\cdotp \cdotp \cdotp \wedge Y_{n-1}}](Z)-\nabla_{ {X_{1\cdots n-1}^\nabla}(Y_1\wedge\cdotp \cdotp \cdotp \wedge Y_{n-1})}Z\notag
		\\
		&:=& X^{\nabla}_{1\cdots n-1}\nabla_{Y_1\wedge\cdotp \cdotp \cdotp \wedge Y_{n-1}}Z -\nabla_{Y_1\wedge\cdotp \cdotp \cdotp \wedge Y_{n-1}}X^{\nabla}_{1\cdots n-1} Z
		-\nabla_{ {X_{1\cdots n-1}^\nabla}(Y_1\wedge\cdotp \cdotp \cdotp \wedge Y_{n-1})}Z,
	\end{eqnarray}
	for all $X_1, \cdots, X_{n-1}, Y_1, \cdots, Y_{n-1}, Z\in \Gamma(A)$ and $n\geqslant 3$.
	When $n=3$, it reads
	\begin{eqnarray*}
		R^\nabla(X_1,X_2,Y_1\wedge Y_2)(Z)&=&X^{\nabla}_{1 2}\nabla_{Y_1\wedge Y_{2}}Z -\nabla_{Y_1\wedge Y_{2}}X^{\nabla}_{12} Z-\nabla_{ {X_{12}^\nabla}(Y_1\wedge Y_{2})}Z
		\\
		&=&[X_1,X_2,\nabla_{Y_1\wedge Y_{2}}Z]^\nabla-\nabla_{Y_1\wedge Y_{2}}[X_1,X_2,Z]^\nabla
		\\
		&&\ -\nabla_{[X_1,X_2,Y_1]^\nabla\wedge Y_2+Y_1\wedge[X_1,X_2,Y_2]^\nabla}Z.
	\end{eqnarray*}
	When $n=4$, the expression of $R^\nabla$ consists of twenty terms. As $n$ gets larger,   more terms are involved.

	It is easy to verify from the defining Eq $\eqref{eq10}$ that the curvature $R^\nabla$ is $C^\infty(M)$-linear with respect to the argument $Y_1\wedge\cdots \wedge Y_{n-1}$.
	However, $R^\nabla$ need not be tensorial in $X_1,\cdots,X_{n-1}$ although it is skew-symmetric in these arguments.
	
	\begin{definition}\label{def2.2}
		A connection $\nabla$ on an $n$-anchored bundle $(A, \rho)$ is called a Filippov connection if the following two conditions are true:
		\begin{enumerate}
			\item[(1)] The curvature $R^\nabla$ is $C^\infty(M)$-linear with respect to its last argument, i.e., for all $f\in C^\infty(M)$ and all $X_1,\cdots,X_{n-1},Y_1,\cdots,Y_n \in \Gamma(A)$, we have
			\begin{equation*}
				R^\nabla(X_1,\cdots,X_{n-1},Y_1\wedge\cdotp\cdotp\cdotp\wedge Y_{n-1})(fY_n) = fR^\nabla(X_1,\cdots,X_{n-1},Y_1\wedge\cdotp\cdotp\cdotp\wedge Y_{n-1})(Y_n);			
			\end{equation*}
			\item[(2)] The following equality holds, to be called the Bianchi-Filippov identity:
			\begin{equation}\label{eq12}	
				0 = \sum_{i=0}^{n-1}(-1)^{(n-1)i}R^\nabla(X_1,\cdots,X_{n-1},Y_{i+1}\wedge \cdotp\cdotp \cdotp \wedge Y_{n}\wedge Y_1\wedge\cdotp\cdotp\cdotp\wedge Y_{i-1})Y_i,	
			\end{equation}
			where $Y_0$ means $Y_n$.
		\end{enumerate}
	\end{definition}
	
	We are ready to state our main theorem, which characterizes Filippov algebroids fully by Filippov connections.
	\begin{theorem}\label{th2.3}
		Let $(A,\rho)$ be an $n$-anchored bundle. If  $\nabla$ is a Filippov connection on $(A,\rho)$, then $(A,\rho,[\cdotp,\cdots,\cdotp]^{\nabla})$ is a Filippov $n$-algebroid, where $[\cdotp,\cdots,\cdotp]^\nabla$ is the   $n$-bracket given by Eq \eqref{eq9}. Moreover, any Filippov $n$-algebroid structure on $(A,\rho)$ arises from a Filippov connection in this way.
	\end{theorem}
	
	\subsection{Proof of Theorem~\ref{th2.3}}
	The proof of Theorem~\ref{th2.3} is divided, and will follow immediately from the three lemmas below.
	\begin{lemma}\label{lemma2.4}
		Let $\nabla$ be a connection on  an $n$-anchored bundle $(A,\rho)$.
		The curvature $R^\nabla$  satisfies the first condition of Definition \ref{def2.2}    if and only if the anchor $\rho$ intertwines the induced $n$-bracket $[\cdotp,\cdots,\cdotp]^\nabla$   and the Lie bracket $[\cdotp,\cdotp]_{TM}$ on $\Gamma(TM)$, i.e.,
		\begin{eqnarray*}
			[\rho(X_1\wedge \cdotp \cdotp \cdotp \wedge X_{n-1}),\rho(Y_1\wedge \cdotp \cdotp \cdotp \wedge Y_{n-1})]_{TM}
			=\sum_{i=1}^{n-1}\rho(Y_1\wedge \cdotp \cdotp \cdotp \wedge [X_1,\cdots,X_{n-1},Y_i]^{\nabla}\wedge\cdotp \cdotp \cdotp \wedge Y_{n-1}),		
		\end{eqnarray*}
		for all $X_1, \cdots, X_{n-1}, Y_1, \cdots, Y_{n-1} \in \Gamma(A)$.
	\end{lemma}
	\begin{proof}
		By the definition of curvature, we have
		\begin{align*}
			&\quad R^\nabla(X_1,\cdots,X_{n-1},Y_1\wedge \cdots \wedge Y_{n-1})(fY_n) \qquad \text{by Eq~\eqref{eq10}}\\
			&= [ {X_{1\cdots n-1}^\nabla},\nabla_{Y_1\wedge\cdots \wedge Y_{n-1}}](fY_n)-\nabla_{ {X_{1\cdots n-1}^\nabla}(Y_1\wedge\cdots \wedge Y_{n-1})}(fY_n) \\
			&= X^{\nabla}_{1\cdots n-1}\nabla_{Y_1\wedge\cdots \wedge Y_{n-1}}(fY_n) -\nabla_{Y_1\wedge\cdots \wedge Y_{n-1}}X^{\nabla}_{1 \cdots n-1}(fY_n)
			-\nabla_{ {X_{1\cdots n-1}^\nabla}(Y_1\wedge\cdots \wedge Y_{n-1})}(fY_n) \\
			&= fR^\nabla(X_1,\cdots,X_{n-1},Y_1\wedge\cdots \wedge Y_{n-1})(Y_n) + \rho(X_1\wedge\cdots \wedge X_{n-1}) \rho(Y_1\wedge\cdots \wedge Y_{n-1})(f)Y_n \\
			&\quad -\rho(Y_1\wedge\cdots \wedge Y_{n-1})\rho(X_1\wedge\cdots \wedge X_{n-1})(f)Y_n -\sum_{i=1}^{n-1}\rho(Y_1\wedge\cdots \wedge X^{\nabla}_{1\cdots n-1}(  Y_i)\wedge \cdots \wedge Y_{n-1})(f)Y_n.
		\end{align*}
		Hence, the curvature $R^\nabla$ is $C^\infty(M)$-linear with respect to its last argument if and only if
		\begin{eqnarray*}
			&&\rho(X_1\wedge\cdots \wedge X_{n-1})\rho(Y_1\wedge \cdots \wedge Y_{n-1})-\rho(Y_1\wedge\cdots \wedge Y_{n-1})\rho(X_1\wedge\cdots \wedge X_{n-1}) \\
			&=&\sum_{i=1}^{n-1}\rho(Y_1\wedge\cdots \wedge X^{\nabla}_{1\cdots n-1}(Y_i) \wedge \cdots \wedge Y_{n-1}) \\
			&=&\sum_{i=1}^{n-1}\rho(Y_1\wedge\cdots \wedge [X_1,\cdots,X_{n-1},Y_i]^\nabla\wedge \cdots \wedge Y_{n-1}).
		\end{eqnarray*}
	\end{proof}
	\begin{lemma}\label{lemma2.5}
		Let $\nabla$ be a connection on  an $n$-anchored bundle $(A,\rho)$.
		The curvature $R^\nabla$  satisfies the second condition of Definition \ref{def2.2}, i.e. the Bianchi-Filippov identity~\eqref{eq12}, if and only if the induced $n$-bracket $[\cdotp,\cdots,\cdotp]^\nabla$ satisfies the (generalized) Jacobi identity \eqref{eq3}.
	\end{lemma}
	\begin{proof}\renewcommand{\qedsymbol}{}
		The statement follows directly from the following lines of computation:	
		\begin{align*}
			&\quad \sum_{i=0}^{n-1}(-1)^{(n-1)i} R^\nabla(X_1,\cdots,X_{n-1},Y_{i+1} \wedge \cdots \wedge Y_{n} \wedge Y_1 \wedge \cdots \wedge Y_{i-1})Y_i    \qquad \text{by Eq.~\eqref{eq10}}       \\
			&= \sum_{i=0}^{n-1}(-1)^{(n-1)i} \left( [X^{\nabla}_{1\cdots n-1}, \nabla_{Y_{i+1} \wedge \cdots \wedge Y_{n} \wedge Y_1 \wedge \cdots \wedge Y_{i-1}}](Y_i) - \nabla_{ {X_{1\cdots n-1}^\nabla}(Y_{i+1} \wedge \cdots \wedge Y_{n} \wedge Y_1 \wedge \cdots \wedge Y_{i-1})}Y_i \right) \\
			&= \left[X_1, \cdots, X_{n-1}, \sum_{i=0}^{n-1}(-1)^{(n-1)i}\nabla_{Y_{i+1} \wedge \cdots \wedge Y_{n} \wedge Y_1 \wedge \cdots \wedge Y_{i-1}}(Y_i)\right]^\nabla \\
			&\quad -  \sum_{i=0}^{n-1}(-1)^{(n-1)i} \left(\nabla_{Y_{i+1} \wedge \cdots \wedge Y_{n} \wedge Y_1 \wedge \cdots \wedge Y_{i-1}} X^{\nabla}_{1\cdots n-1}(Y_i)  + \sum_{j\neq i} \nabla_{Y_{i+1} \wedge \cdots \wedge {X_{1\cdots n-1}^\nabla}(Y_{j}) \wedge \cdots \wedge Y_{i-1}}Y_i\right) \quad \text{by Eq.~\eqref{eq9}} \\
			&=  [X_1,\cdots,X_{n-1},[Y_1,\cdots,Y_n]^\nabla]^\nabla -  \sum_{i=0}^{n-1}(-1)^{(n-1)i} \nabla_{Y_{i+1} \wedge \cdots \wedge Y_{n} \wedge Y_1 \wedge \cdots \wedge Y_{i-1}} [X_1,\cdots, X_{n-1}, Y_i]^\nabla \\
			&\quad - \sum_{i=0}^{n-1} \sum_{j\neq i} (-1)^{(n-1)i} \nabla_{Y_{i+1} \wedge \cdots \wedge [X_1, \cdots, X_{n-1}, Y_{j}]^\nabla \wedge \cdots \wedge Y_{i-1}}Y_i \quad \text{by Eq.~\eqref{eq9}} \\
			&= [X_1,\cdots,X_{n-1},[Y_1,\cdots,Y_n]^\nabla]^\nabla - \sum_{i=1}^{n}[Y_1,\cdotp, [X_1,\cdots,X_{n-1},Y_i]^\nabla, \cdots,Y_n]^\nabla.
		\end{align*}
	\end{proof}
	
	The next lemma shows that any Filippov algebroid can be realized by a Filippov connection.
	\begin{lemma}\label{lemma2.6}
		Let $(A, [\cdotp,\cdots,\cdotp], \rho)$ be a Filippov $n$-algebroid. Then there exists a Filippov connection $\nabla$   on the underlying $n$-anchored bundle $(A,\rho)$ such that $[\cdotp,\cdots,\cdotp]=[\cdotp,\cdots,\cdotp]^{\nabla}$ (the torsion-free property).
	\end{lemma}	
	\begin{proof}\renewcommand{\qedsymbol}{}
		Given a connection $\nabla^\circ$ on $(A,\rho)$,  we are able to obtain an $\mathbb{R}$-multilinear operation $K(\cdotp,\cdotp \cdotp \cdotp,\cdotp)$ on $\Gamma(A)$ by
		\begin{eqnarray*}
			K(X_1,\cdots,X_n) &:=&[X_1,\cdots,X_n]- [X_1,\cdots,X_n]^{\nabla^\circ}.
		\end{eqnarray*}
		Using axioms of Filippov algebroids, it is easy to see that $K(\cdotp,\cdots,\cdotp)$ is indeed $C^\infty(M)$-multilinear. 		
		Then we define a new connection $\nabla$ on $(A,\rho)$ by
		\begin{eqnarray*}
			\nabla_{X_1\wedge \cdots \wedge X_{n-1}}X_n &:=& \frac{1}{n}K(X_1,\cdots,X_n)+\nabla^\circ_{X_1\wedge \cdots \wedge X_{n-1}}X_n.
		\end{eqnarray*}
		It remains to check the desired equality:
		\begin{eqnarray*}
			[X_1,\cdots,X_n ]^{\nabla}
			&=&\nabla_{X_1\wedge \cdots \wedge X_{n-1}}X_n + \sum_{i=1}^{n-1}(-1)^{(n-1)i} \nabla_{X_{i+1}\wedge \cdots \wedge X_{n}\wedge X_1\wedge \cdots \wedge X_{i-1}}X_i \\
			&=&K(X_1,\cdots,X_n)+\nabla^\circ_{X_1\wedge \cdots \wedge X_{n-1}}X_n + \sum_{i=1}^{n-1}(-1)^{(n-1)i}\nabla^\circ_{X_{i+1}\wedge \cdots \wedge X_{n}\wedge X_1\wedge \cdots \wedge X_{i-1}}X_i  \\
			&=&[X_1,\cdots,X_n].
		\end{eqnarray*}
	\end{proof}	
	
	The following examples illustrate three Filippov connections on the trivial tangent bundle $T\mathbb{R}^m$ of $\mathbb{R}^m$ for $m  \geqslant n\geqslant 2$.
	\begin{example}
		Consider the trivial $n$-anchored vector bundle $(T\mathbb{R}^m, \rho =0)$. Suppose that the vector space $\mathbb{R}^m$ is endowed with a Filippov $n$-algebra structure whose structure constants are $\{c^j_{i_1\cdots i_{n}}\}$ with respect to the standard basis of $\mathbb{R}^m$.
		Given a smooth function $g\in C^\infty(\mathbb{R}^m)$ and a set of constants $\{a^j_{i_1\cdots i_{n-1};i_n}\}$ satisfying the equality:
		\begin{equation}\label{eq19}
			a^j_{i_1\cdots i_{n-1};i_n}+\sum_{k=1}^{n-1}(-1)^{(n-1)k}a^j_{i_{k+1}\cdots i_n i_1\cdots i_{k-1};i_k}=c^j_{i_1\cdots i_{n}},
		\end{equation}
		we are able to obtain a connection on $(T\mathbb{R}^m,\rho=0)$ generated by the only one nontrivial relation:
		\begin{equation*}\label{eq 20}
			\nabla_{\frac{\partial}{\partial x_{i_1}}\wedge\cdots\wedge\frac{\partial}{\partial x_{i_{n-1}}}}\frac{\partial}{\partial x_{i_n}}:=g\sum_{j=1}^{m}a^j_{i_1\cdots i_{n-1};i_n}\frac{\partial}{\partial x_j}.
		\end{equation*}
		It follows from the recipe in Eqs \eqref{eq9} and \eqref{eq19} that
		\begin{equation*}
			\left[\frac{\partial}{\partial x_{i_1}},\cdots,\frac{\partial}{\partial x_{i_n}}\right]^{\nabla}=g\sum_{i=1}^{m}c^j_{i_1\cdots i_{n}}\frac{\partial}{\partial x_j}.
		\end{equation*}
		So, what we recover is the Filippov structure on $T\mathbb{R}^m$ as in Example \ref{example1.3}. Hence, $\nabla$ is indeed a Filippov connection.
	\end{example}
	
	\begin{example}
		Consider the $n$-anchor map $\rho$ on the tangent bundle $T\mathbb{R}^m$ defined by the tensor field  $dx_1\wedge \cdotp \cdotp \cdotp \wedge dx_{n-1} \otimes \frac{\partial}{\partial x_1}$, where $x_1,\cdots,x_{n-1},x_n,\cdots,x_m$ are coordinate functions of $\mathbb{R}^m$.
		It is obvious that the (nontrivial) connection on $(T\mathbb{R}^m,\rho)$
		generated by the trivial relation:
		\[
		\nabla_{\frac{\partial}{\partial x_{i_1}}\wedge\cdots\wedge\frac{\partial}{\partial x_{i_{n-1}}}}\frac{\partial}{\partial x_{i_n}}:=0,
		\]
		produces  a nontrivial $n$-bracket which is compatible with $\rho$.
		Indeed, what we recover is the Filippov structure on $T\mathbb{R}^m$ as in Example \ref{example1.4}, and the said   connection $\nabla $ is a Filippov connection.
	\end{example}
	\begin{example}
		Continue to work with the anchored bundle $(A,\rho)$ as in the previous example. We consider a different connection with the only nontrivial generating relations:
		\begin{equation*}
			\nabla_{Z}\frac{\partial}{\partial x_{k}}=	\left\{
			\begin{aligned}
				&(-1)^\sigma\frac{\partial}{\partial x_{k}}, \mbox{ if }Z\quad  = \frac{\partial}{\partial x_{\sigma_1}}\wedge\cdots\wedge\frac{\partial}{\partial x_{\sigma_{n-1}}}, \\
				& 0,\mbox{ otherwise,}
			\end{aligned}
			\right.
		\end{equation*}
		where $\sigma$ is a permutation $\{1,\cdots,n-1\}$,   for all $x_k\in \{x_1,\cdots,x_m\}$. Then, the associated $n$-bracket is given by
		\begin{equation*}
			\left[\frac{\partial}{\partial x_{\sigma_1}},\cdots,\frac{\partial}{\partial x_{\sigma_{n-1}}},\frac{\partial}{\partial x_{k}}\right]^{\nabla} = \left\{
			\begin{aligned}
				&(-1)^\sigma\frac{\partial}{\partial x_{k}}, \quad \text{if}\quad k>n-1,
				\\
				&0,\text{otherwise},
			\end{aligned}
			\right.
		\end{equation*}
		where $\sigma$ is a permutation $\{1,\cdots,n-1\}$.
		By subtle analysis, one can find that the associated curvature $R^{\nabla}$ is just zero. Hence $\nabla$ is truly a Filippov connection and the above bracket defines a Filippov algebroid structure on $(A,\rho)$.
	\end{example}
	
	\subsection{Construction of Filippov connections}\label{Sec:Dxipair}
	Let $A\to M$ be a vector bundle. Consider the bundle $\CDO(A)$ of covariant differential
	operators (cf. \cite{Lie}*{III}, \cite{pseudoalgebras}, see also \cite{Differential}, where the notation $\mathcal{D}(A)$ is used instead of $\CDO(A)$). An element $D$ of $\Gamma(\CDO(A))$, called a covariant differential operator, is  an $\R$-linear operator $\Gamma(A)\rightarrow \Gamma(A)$ together with a vector field $ \hat{D} \in \Gamma(TM)$, called the symbol of $D$, satisfying
	\begin{eqnarray*}
		D(fX)=fD(X)+\hat{D}(f)\cdot X, \qquad\forall X\in\Gamma(A), f\in\CinfM.
	\end{eqnarray*}
	The operator $D$ can be first extended by the Leibniz rule to an operator $D \colon \Gamma(\wedge^{n-1} A) \to \Gamma(\wedge^{n-1}A)$. By taking dual we obtain an operator $D \colon \Gamma(\wedge ^{n-1} A^*)\rightarrow \Gamma(\wedge ^{n-1} A^*)$ defined by
	\begin{eqnarray}\label{dual operator}
		\langle X_1\wedge\cdotp\cdotp\cdotp\wedge X_{n-1}|D(\bar{\eta})\rangle
		=\hat{D}\langle X_1\wedge\cdotp\cdotp\cdotp\wedge X_{n-1}|\bar{\eta}\rangle-\sum_{i=1}^{n-1}\langle X_1\wedge\cdotp\cdotp\cdotp\wedge D(X_i)\wedge\cdotp\cdotp\cdotp\wedge X_{n-1}|\bar{\eta}\rangle,
	\end{eqnarray}
	for all $X_1, \cdots, X_{n-1}\in \Gamma(A)$ and $\bar{\eta} \in\Gamma(\wedge^{n-1}A^*)$.
	
	Given a pair $(D,\bar{\xi})$, where $D\in\Gamma(\CDO(A))$  and $\bar{\xi}\in \Gamma(\wedge^{n-1}A^*)$, one is able to construct a map
	\begin{eqnarray*}
		\rho^{(D,\bar{\xi})} \colon  \Gamma(\wedge^{n-1}A) &\to& \Gamma(TM),\\
		X_1\wedge\cdotp\cdotp\cdotp\wedge X_{n-1} &\mapsto&  \langle X_1\wedge\cdotp\cdotp\cdotp\wedge X_{n-1}|\bar{\xi} \rangle \hat{D}.
	\end{eqnarray*}
	It is clear that $\rho^{(D,\bar{\xi})}$ makes $A$ an $n$-anchored bundle, and the rank of the image of $\rho^{(D,\bar{\xi})}$ does not exceed~1.
	
	Define a connection on the $n$-anchored bundle $(A,\rho^{(D,\bar{\xi})})$ by
	\begin{equation}\label{Eqt:pairconnection}
		\nabla^{(D,\bar{\xi})}_{X_1\wedge\cdotp\cdotp\cdotp\wedge X_{n-1}}X_n : =\langle X_1\wedge \cdots \wedge X_{n-1}|\bar{\xi}\rangle D(X_n).
	\end{equation}
	\begin{proposition}
		If the pair $(D,\bar{\xi})$ is subject to the relation
		\begin{equation}\label{EigenCondition}
			D(\bar{\xi})=g\bar{\xi} ,  \quad \mbox{ for some }  g\in C^\infty(M),
		\end{equation}
		then $\nabla^{(D,\bar{\xi})}$ defined as in~\eqref{Eqt:pairconnection} is a Filippov connection on the $n$-anchored bundle $(A,\rho^{(D,\bar{\xi})})$.
	\end{proposition}	
	\begin{proof}\renewcommand{\qedsymbol}{}
		We denote $\nabla^{(D,\bar{\xi})}$ by $\nabla$ for simplicity below. It suffices to check the associated curvature $R^\nabla$ is $C^\infty(M)$-linear with respect to the last argument and satisfies the Bianchi-Filippov identity~\eqref{eq12}. In fact, we have
		\begin{align*}
			&\quad R^\nabla(X_1,\cdots,X_{n-1},Y_1\wedge\cdots \wedge Y_{n-1})(fY_n) \qquad \text{by Eq.~\eqref{eq10}}\\
			&= X^{\nabla}_{1\cdots n-1}\nabla_{Y_1\wedge\cdots \wedge Y_{n-1}}(fY_n) -\nabla_{Y_1\wedge\cdots \wedge Y_{n-1}}X^{\nabla}_{1\cdots n-1}(fY_n)
			-\nabla_{ {X_{1\cdots n-1}^\nabla}(Y_1\wedge\cdots \wedge Y_{n-1})}(fY_n) \; \text{by Eqs.~\eqref{eq9}, \eqref{dual operator}, \eqref{Eqt:pairconnection}} \\
			&= fR^\nabla(X_1,\cdots,X_{n-1}, Y_1 \wedge\cdots\wedge\wedge Y_{n-1})(Y_n)+(\langle X_1 \wedge \cdots \wedge X_{n-1} | \bar{\xi} \rangle \\
			&\quad  \langle Y_1\wedge\cdots\wedge Y_{n-1} | D(\bar{\xi})  \rangle -\langle Y_1\wedge\cdots\wedge Y_{n-1}|\bar{\xi} \rangle \langle X_1\wedge\cdots\wedge X_{n-1}|D(\bar{\xi})  \rangle)\hat{D}(f)Y_n \;\;\text{by Eq.~\eqref{EigenCondition}} \\
			&= f R^\nabla(X_1,\cdots,X_{n-1},Y_1\wedge\cdots \wedge Y_{n-1})(Y_n),
		\end{align*}
		and 	
		\begin{align*}
			&\quad \sum_{i=0}^{n-1}(-1)^{(n-1)i} R^\nabla(X_1,\cdots,X_{n-1},Y_{i+1} \wedge \cdots \wedge Y_{n} \wedge Y_1 \wedge \cdots \wedge Y_{i-1})Y_i     \qquad \text{by Eq.~\eqref{eq10}}      \\
			&= \sum_{i=0}^{n-1}(-1)^{(n-1)i}  [X^{\nabla}_{1\cdots n-1}, \nabla_{Y_{i+1} \wedge \cdots \wedge Y_{n} \wedge Y_1 \wedge \cdots \wedge Y_{i-1}}](Y_i) \\
			&\quad  -\sum_{i=0}^{n-1}(-1)^{(n-1)i} \nabla_{ {X_{1\cdots n-1}^\nabla}(Y_{i+1} \wedge \cdots \wedge Y_{n} \wedge Y_1 \wedge \cdots \wedge Y_{i-1})}Y_i  \quad \text{by Eqs.~\eqref{eq9} and \eqref{Eqt:pairconnection} }\\
			&=\sum_{i=0}^{n-1}(-1)^{(n-1)i}\langle Y_{i+1} \wedge \cdots \wedge Y_{n} \wedge Y_1 \wedge \cdots \wedge Y_{i-1}|\bar{\xi}  \rangle X^{\nabla}_{1\cdots n-1}(D(Y_i)) \\
			&\quad  -\sum_{i=0}^{n-1}(-1)^{(n-1)i}\langle Y_{i+1} \wedge \cdots \wedge Y_{n} \wedge Y_1 \wedge \cdots \wedge Y_{i-1}|\bar{\xi}  \rangle D(X^{\nabla}_{1\cdots n-1}(Y_i)) \\
			&\quad  +\sum_{i=0}^{n-1}(-1)^{(n-1)i}X^{\nabla}_{1\cdots n-1}(\langle Y_{i+1} \wedge \cdots \wedge Y_{n} \wedge Y_1 \wedge \cdots \wedge Y_{i-1}|\bar{\xi} \rangle )D(Y_i) \\
			&\quad  -\sum_{i=0}^{n-1}(-1)^{(n-1)i}\sum_{j\neq i}\langle Y_{i+1} \wedge \cdots \wedge {X_{1\cdots n-1}^\nabla}(Y_{j}) \wedge \cdots \wedge Y_{i-1}|\bar{\xi}\rangle D(Y_i) \;\; \text{by Eqs.\eqref{eq9},\eqref{dual operator} and \eqref{Eqt:pairconnection}}\\
			&= \sum_{i=0}^{n-1}(-1)^{(n-1)i} \langle X_1\wedge\cdots\wedge X_{n-1}|\bar{\xi} \rangle \langle Y_{i+1} \wedge \cdots \wedge Y_{n} \wedge Y_1 \wedge \cdots \wedge Y_{i-1}|D(\bar{\xi}) \rangle\\
			&\quad  - \sum_{i=0}^{n-1}(-1)^{(n-1)i} \langle  Y_{i+1} \wedge \cdots \wedge Y_{n} \wedge Y_1 \wedge \cdots \wedge Y_{i-1}|\bar{\xi} \rangle \langle X_1\wedge\cdots\wedge X_{n-1}|D(\bar{\xi}) \rangle  \;\; \text{by Eq.~\eqref{EigenCondition}}\\
			& = 0.
		\end{align*}
		Hence, $\nabla^{(D,\bar{\xi})}$ defined as in~\eqref{Eqt:pairconnection} is indeed a Filippov connection.
	\end{proof}	
	As a consequence of Theorem \ref{th2.3}, a pair $(D,\bar{\xi})$ subject to Condition \eqref{EigenCondition} produces a Filippov $n$-algebroid structure on $A$. Its $n$-bracket reads:
	\begin{eqnarray*}
		[X_1,\cdots,X_n]^{\nabla^{(D,\bar{\xi})}}
		&=& \langle X_1\wedge\cdotp\cdotp\cdotp\wedge X_{n-1}|\bar{\xi}\rangle D(X_n) +(-1)^{n-1}\langle X_2\wedge \cdotp \cdotp \cdotp \wedge X_{n}|\bar{\xi}\rangle D(X_1)+ \cdots \\
		&&\ +\langle X_{n-1}\wedge X_n\wedge X_1\wedge \cdotp \cdotp \cdotp \wedge X_{n-3}|\bar{\xi}\rangle D(X_{n-2})     +(-1)^{n-1}\langle X_{n}\wedge X_1\wedge \cdotp \cdotp \cdotp \wedge X_{n-2}|\bar{\xi}\rangle D(X_{n-1}).
	\end{eqnarray*}
	
	\section{A construction of linear Nambu-Poisson structures}\label{Sec:Nambupart}

	In this section, we unravel under certain conditions a relation between linear Nambu-Poisson structures and Filippov connections.

	\subsection{From Filippov algebroids to linear Nambu-Poisson structures}
	\begin{definition}\cite{Tak, ref15}
		A Nambu-Poisson structure of order $n$ on a smooth manifold $P$ is  an $\mathbb{R}$-multilinear and skew-symmetric $n$-bracket on the smooth function space $C^\infty(P)$:
		\[
		\{\cdotp,\cdotp\cdotp\cdotp,\cdotp\}\colon \underbrace{C^\infty(P)\times\cdotp\cdotp\cdotp\times C^\infty(P)}_{n-\mbox{copies}}\to C^\infty(P),
		\]
		satisfying the following two conditions:
		\begin{enumerate}
			\item[(1)] The $n$-bracket is a derivation with respect to $C^\infty(P)$-multiplications:			
			\[
			\{f_1,\cdots,f_{n-1},g_1g_2\}=g_1\{f_1,\cdots,f_{n-1},g_2\}+\{f_1,\cdots,f_{n-1},g_1\}g_2;
			\]
			\item[(2)] The (generalized) Jacobi identity (also known as the fundamental identity)
			\[ \{f_1,\cdots,f_{n-1},\{g_1,\cdots,g_n\}\}=\sum_{i=1}^{n}\{g_1,\cdots,\{f_1,\cdots,f_{n-1},g_i\},\cdots,g_{n}\},
			\]
			holds for all $f_{i}$ and $g_{j} \in C^\infty(P)$.
		\end{enumerate}
		The pair $(P, \{\cdotp,\cdotp\cdotp\cdotp,\cdotp\})$ is called a Nambu-Poisson manifold.
	\end{definition}

	Alternatively, one could express the said bracket via an $n$-vector field $\pi$ on $P$ such that
	\begin{equation}\label{Eqt:piandnbracket}
		\{f_1,\cdots, f_n\}=\pi(df_1,\cdots, df_n),\quad \forall f_1,\cdots, f_n\in C^\infty(P).
	\end{equation}	
	
	Given a smooth vector bundle  $p\colon A\to M$, the section space $\Gamma(A )$   are identified as the space $C^\infty_{\mathrm{lin}}(A^*)$ of fiberwise linear functions on $A^*$, the dual vector bundle of $A$; while elements in $p^\ast(C^\infty(M))$ are called basic functions on $A^*$.
	To fix the notations, 	for any section $X \in \Gamma(A)$, let $\phi_X \in C^\infty_{\mathrm{lin}}(A^*)$ be the corresponding linear function on $A^\ast$.
	
	\begin{definition}\cite{BBDM}
		A Nambu-Poisson structure of order $n$ on the vector bundle $ A^\ast \to M$ is said to be linear, if it satisfies the following three conditions:
		\begin{enumerate}
			\item[(1)] The bracket of $n$ linear functions is again a linear function;
			\item[(2)] The bracket of $(n-1)$ linear functions and a basic function is a basic function;
			\item[(3)] The bracket of $n$ functions is zero if there are more than one basic functions among the arguments.
		\end{enumerate}
	\end{definition}
	In fact, the second and the third condition in the above definition can be derived from the first condition.
	
	Note that any Poisson manifold is a Nambu-Poisson manifold of order $2$.
	A well-known fact is the following: A Lie algebroid $A$ over $M$ gives rise to a linear Poisson manifold $A^*$, and vice versa.   It is pointed out in \cite{BBDM}*{Theorem 4.4} that a linear Nambu-Poisson structure of order $n$ on $A^*$ corresponds to a Filippov $n$-algebroid structure on $A$ (see also~\cite{BL}). However, the reverse process is generally not valid for the cases of $n\geqslant 3$, mainly because the condition of a Nambu-Poisson structure is very strong (cf. \cite{ref13}). Nevertheless, in this paper, we will require that $A$ be a rank $n$ vector bundle and establish the one-to-one correspondence between Filippov $n$-algebroid structures on $A$ and linear Nambu-Poisson structures on $A^*$. In specific, under the said condition, our main theorem below serves as a complement to \cite{BBDM}*{Theorem 4.4}.
	\begin{theorem}\label{prop3.3}
		Let $(A,\rho,[\cdotp,\cdots,\cdotp])$ be a Filippov $n$-algebroid over a  smooth manifold $M$, where $A  \to M$ is a vector bundle of rank $n\geqslant 3$. Then there exists a unique linear Nambu-Poisson structure on the dual bundle $A^*\to M$ such that for all sections $X_1,\cdots,X_n\in \Gamma(A)$,
		\begin{equation}\label{Nambu condition-1}
			\{\phi_{X_1},\cdots, \phi_{X_n}\}=\phi_{[X_1,\cdots, X_n]}.
		\end{equation}
	\end{theorem}
	
	Note that, if Eq \eqref{Nambu condition-1} holds, then it is easy to deduce that the linear    Nambu-Poisson structure on $A^*$ and the anchor  map $\rho$ are also related:
	\begin{equation}\label{Nambu condition-2}
		\{\phi_{X_1},\cdots,\phi_{X_{n-1}},p^\ast f \}=p^\ast(\rho(X_1\wedge\cdotp\cdotp\cdotp \wedge X_{n-1})(f)),
	\end{equation}
	for all $f\in C^\infty(M)$ and $X_1, \cdots, X_{n-1} \in \Gamma(A)$.

	As a direct application of Theorems \ref{th2.3} and \ref{prop3.3}, one can construct linear Nambu-Poisson structures	 out of Filippov connections:
	\begin{Cor}\label{propo3.3}
		If $\nabla$ is a Filippov connection on an  $n$-anchored bundle $(A, \rho)$, where $A \to M$ is a vector bundle of rank $n\geqslant 3$, then the dual bundle $A^\ast$ admits a unique linear Nambu-Poisson structure of order $n$ defined by
		\begin{equation*}
			\{\phi_{X_1},\cdots, \phi_{X_n}\}=\phi_{[X_1,\cdots, X_n]^\nabla},
		\end{equation*}		
		for all $X_1,\cdots,X_n\in \Gamma(A)$.
	\end{Cor}

	\subsection{Proof of Theorem \ref{prop3.3}}
	The proof is divided into three  steps.
	
	\textbf{Step 1.} Since functions of type $\phi_{X}$ (for $X\in \Gamma(A)$) and $p^*f$ (for $f\in C^\infty(M)$)  generate $C^\infty(A^*)$, there exists a unique $\mathbb{R}$-multilinear $n$-bracket $\{\cdotp,\cdots,\cdotp\}$ on $C^\infty(A^*)$  satisfying  Eqs \eqref{Nambu condition-1} and \eqref{Nambu condition-2}. We wish to write the corresponding $n$-vector field $\pi$ on $A^*$ explicitly.
	
	To this end, we work locally and consider the trivialization $A|_U\cong U\times \mathbb{R}^n$ over an open subset $U\subset M$ with coordinates $x^1, \cdots, x^m$; let $\{ X_1,\cdots,X_n\}$ be a local  basis  of $\Gamma(A|_U)$. Then $$\{ y_1=\phi_{X_1},\cdots,y_n=\phi_{X_n},p^*x^1,\cdots, p^*x^m\}$$
	forms a chart on $A^*|_U$. For convenience,  $p^*x^i$ is denoted by $x^i$.
	
	Suppose further that the Filippov algebroid $A|_U$ is described by the structure functions  $c^k $ and $f_{1\cdots \widehat{l}\cdots n}\in C^\infty(U)$ such that
	\[
	[X_1,\cdots,X_n]=\sum_{k=1}^n c^k X_k,
	\]
	\[\mbox{and } \quad \rho(X_{1}\wedge\cdots \widehat{X_{l}}\cdots\wedge X_{n})=
	f_{1\cdots \widehat{l}\cdots n}\frac{\partial}{\partial x^1}.
	\]
	Here we have utilized Proposition \ref{propo1.5}. Then one is able to find the expression of the $n$-vector field $\pi$ on $A^*|_U$:
	\begin{equation}\label{n-vector}
		\pi=\sum_{k=1}^nc^ky_k\frac{\partial}{\partial y_1}\wedge\cdots\wedge\frac{\partial}{\partial y_n}
		+\sum_{l=1}^nf_{1\cdots \widehat{l}\cdots n}\frac{\partial}{\partial y_{1}}
		\wedge\cdots\wedge\frac{\partial}{\partial \widehat{y_l}}\wedge\cdots\wedge\frac{\partial}{\partial y_{n}}\wedge \frac{\partial}{\partial x^1}
	\end{equation}
	which corresponds to the $n$-bracket $\{\cdots\}$ on $A^*|_U$.
	
	\textbf{Step 2.} We need to set up a preparatory lemma.
	\begin{lemma}\label{Lem:structureconditions}
		There exists a local basis $\{X_1,\cdots,X_{n}\}$ of $\Gamma(A|_U)$ such that	the corresponding structure functions $c^k$ and $f_{1\cdots \widehat{l}\cdots n}$  satisfy the following relations:  for all $i\neq j$ (in $\{1,\cdots, n\}$),
		\begin{eqnarray}
			f_{1\cdots \widehat{i}\cdots n}\frac{\partial f_{1\cdots \widehat{j}\cdots n}}{\partial x^1}&=&f_{1\cdots \widehat{j}\cdots n}\frac{\partial f_{1\cdots \widehat{i}\cdots n}}{\partial x^1};\label{volume-1}\\
			(-1)^if_{1\cdots \widehat{i}\cdots n}c^j&=&(-1)^jf_{1\cdots \widehat{j}\cdots n}c^i.\label{volume-2}
		\end{eqnarray}	
	\end{lemma}
	\begin{proof}\renewcommand{\qedsymbol}{}
		Consider the map $\rho\colon \wedge^{n-1}A|_U \to TU$. By Proposition~\ref{propo1.5}, for any point $p\in U$, we have $\operatorname{rank}(\rho(\wedge^{n-1}A)_p)\leqslant 1$,  and thus $\dim(\ker(\rho_p)) \geqslant (n-1)$. Note that the subset $V \subset U$ where $\dim(\ker(\rho_p))$ is locally constant is open and dense. By a continuity argument if necessary,  we may assume that $\dim(\ker(\rho_p))$ is locally constant on $U$. Thus,  we are able to find  a local basis  $\{Z_1,\cdots,Z_n\}$   of $\Gamma(\wedge^{n-1}A|_U)$ such that $\rho(Z_2)$, $\cdots$, $\rho(Z_n)$ are trivial.
		
		Take an arbitrary $\Omega \in \Gamma(\wedge^{n}A |_U)$ which is nowhere vanishing   on $U$. Consider
		\[
		\Omega ^\sharp \colon A^\ast |_U\to \wedge^{n-1}A |_U,\ \ \ \ \Omega ^\sharp(\alpha):=i_\alpha\Omega ,
		\]
		which is an isomorphism  of vector bundles.    Then we obtain a basis $\{ \alpha_1$, $\cdots$, $\alpha_n \}$ of   $\Gamma(A^\ast |_U)$ by setting $\alpha_i:=(\Omega^\sharp)^{-1}(Z_i)$. Let $\{ X_1,\cdots,X_n\}$ be the  dual basis of $\Gamma(A |_U)$ corresponding to $\{\alpha_1,\cdots,\alpha_n\}$.  There exists a nowhere vanishing smooth function $g\in C^\infty(U)$ such that $\Omega =gX_1\wedge\cdots\wedge X_n$, and hence $Z_i=i_{\alpha_i}\Omega =gX_1\wedge\cdots\widehat{X_i}\wedge\cdots\wedge X_n$.	
		
		Since we have $\rho(Z_i)=0$ ($\forall i\in \{2,\cdots, n\}$), we also have
		\begin{equation}
			\label{Eqt:rhoXizero} \rho(X_1\wedge\cdots\widehat{X_i}\wedge\cdots\wedge X_n)=0,\quad \forall i\in \{2,\cdots, n\}.
		\end{equation}		
		Using the axiom of a Filippov algebroid, we have a relation
		\begin{align*}
			&\quad f_{1\cdots\widehat{i}\cdots n}\frac{\partial f_{1\cdots\widehat{j}\cdots n}}{\partial x^1}\frac{\partial }{\partial x^1}-f_{1\cdots\widehat{j}\cdots n}\frac{\partial f_{1\cdots\widehat{i}\cdots n}}{\partial x^1}\frac{\partial }{\partial x^1} \\
			&=[f_{1\cdots\widehat{i}\cdots n}\frac{\partial}{\partial x^1},f_{1\cdots\widehat{j}\cdots n}\frac{\partial}{\partial x^1}]_{TM} \\
			&=[\rho(X_1\wedge\cdots\wedge\widehat{X_i}\wedge\cdots\wedge X_n),\rho(X_1\wedge\cdots\wedge\widehat{X_j}\wedge\cdots\wedge X_n)]_{TM} \quad \text{by Eq. \eqref{eq1}} \\
			&= \sum_{k=1,k< j}^{n} \rho(X_1\wedge\cdots\wedge[X_1,\cdots,\widehat{X_i},\cdots,X_n,X_k] \wedge\cdots\wedge \widehat{X_j}\wedge\cdots X_n) \\
			&\quad  +\sum_{k=1,k> j}^{n}\rho(X_1\wedge\cdots\wedge \widehat{X_j}\wedge\cdots \wedge [X_1,\cdots, \widehat{X_i},\cdots,X_n,X_k] \wedge \cdots X_n) \\
			&= (-1)^{(n-i)}c^if_{1\cdots\widehat{j}\cdots n}\frac{\partial }{\partial x^1}-(-1)^{(n-j)}c^jf_{1\cdots\widehat{i}\cdots n}\frac{\partial }{\partial x^1}.
		\end{align*}
		According to the previous fact \eqref{Eqt:rhoXizero}, all the lines above must be trivial, and thus the desired two equalities \eqref{volume-1} and \eqref{volume-2} are proved.
	\end{proof}
	
	\textbf{Step 3.}
	We wish to show that the $n$-bracket $\{\cdotp,\cdots,\cdotp\}$ given in Step 1, or the $n$-vector field $\pi$ locally given in Eq \eqref{n-vector}, is a linear Nambu-Poisson structure on $A^*$.
	
	We need the following proposition due to Dufour and Zung~\cite{DZ}.
	\begin{proposition}\cite{DZ}\label{Omega propo}
		Let $\Omega$ be a volume form on an $l$-dimensional manifold $P$, and $\pi$ an   $n$-vector filed on $P$, where $l>n\geqslant 3$.  Consider  the $ (l-n)$-form  $\omega:=\iota_\pi\Omega$ on $P$. Then $\pi$ defines a Nambu-Poisson structure (via Eq \eqref{Eqt:piandnbracket}) if and only if $\omega$ satisfies the following two conditions:
		\begin{eqnarray}
			(\iota_K\omega)\wedge\omega=0, \label{NP-1}
			\\
			(\iota_K\omega)\wedge d\omega=0, \label{NP-2}
		\end{eqnarray}
		for any $(l-n-1)$-vector field $K$ on $P$.
	\end{proposition}
	
	Consider the volume form $\Omega= dy_1\wedge\cdots dy_n \wedge dx^1\wedge\cdots\wedge dx^m$   on $A^\ast|_U$, where $U$, $y_i$, and $x^j$ are as earlier, and we suppose that such a coordinate system stems from $\{X_1,\cdots, X_n\}$ fulfills Lemma \ref{Lem:structureconditions}. According to Proposition \ref{Omega propo}, we need to examine the $m$-form defined by:
	\begin{equation*}\label{eq:omega}
		\omega:=\iota_{\pi}\Omega=\sum_{k=1}^nc^ky_kdx^1\wedge\cdots\wedge dx^m+\sum_{j=1}^n(-1)^{n-j+1}f_{1\cdots \widehat{j}\cdots n}dy_j\wedge dx^2\wedge\cdots\wedge dx^m.	
	\end{equation*}	
	We can easily  check that    $\omega$ satisfies Eq \eqref{NP-1} and hence it remains to check Eq \eqref{NP-2}. One first finds that
	\begin{eqnarray*}\label{eq:domega}
		d\omega&=&\sum_{k=1}^ny_kdc^k\wedge dx^1\wedge\cdots\wedge dx^m+\sum_{k=1}^nc^kdy_k\wedge dx^1\wedge\cdots\wedge dx^m\notag\\
		&&\ +\sum_{j=1}^n(-1)^{n-j+1}df_{1\cdots \widehat{j}\cdots n}\wedge dy_j\wedge dx^2\wedge\cdots\wedge dx^m\notag\\
		&=&\sum_{k=1}^nc^kdy_k\wedge dx^1\wedge\cdots\wedge dx^m+\sum_{j=1}^n(-1)^{n-j+1}\frac{\partial f_{1\cdots\widehat{j}\cdots n}}{\partial x^1}dx^1\wedge dy_j\wedge dx^2\wedge\cdots\wedge dx^m.
	\end{eqnarray*}	
	
	Consider the following special type of $(m-1)$-vector field on $A^\ast|_U$:  $K =\frac{\partial}{\partial x^2}\wedge\cdots\wedge\frac{\partial}{\partial x^m}$.    Then one computes:
	\begin{align*}
		(\iota_K \omega) \wedge d \omega
		&= (-1)^{m-1} \left(\sum_{k=1}^nc^ky_kdx^1+ \sum_{j=1}^n(-1)^{n-j+1}f_{1\cdots \widehat{j}\cdots n} dy_j \right) \wedge d\omega \\
		&= \sum_{j=1}^n (-1)^{m+n-j}f_{1\cdots \widehat{j}\cdots n}dy_j \wedge \sum_{i=1}^nc^idy_i\wedge dx^1\wedge\cdots\wedge dx^m\\
		&\quad +\sum_{j=1}^n(-1)^{m+n-j}f_{1\cdots \widehat{j}\cdots n}dy_j\wedge \sum_{i=1}^n(-1)^{n-i+1}\frac{\partial f_{1\cdots\widehat{i}\cdots n}}{\partial x^1}dx^1\wedge dy_i\wedge dx^2\wedge\cdots\wedge dx^m \\
		&= (-1)^{m}\sum_{j=1}^n\sum_{i=1,i\neq j}^{n}((-1)^{(n-i)}c^if_{1\cdots\widehat{j}\cdots n}-(-1)^{(n-j)}c^jf_{1\cdots\widehat{i}\cdots n})dy_j\wedge dy_i\\
		&\quad +\sum_{j=1}^n\sum_{i=1,i\neq j}^{n}(-1)^{m+i+j}(f_{1\cdots \widehat{j}\cdots n}\frac{\partial f_{1\cdots\widehat{i}\cdots n}}{\partial x^1}-f_{1\cdots \widehat{i}\cdots n}\frac{\partial f_{1\cdots\widehat{j}\cdots n}}{\partial x^1})dy_j\wedge dy_i \;\;\text{by Eqs.~\eqref{volume-1} and \eqref{volume-2}} \\
		&= 0.
	\end{align*}
	This justifies Eq~\eqref{NP-2} for this particular $K$. For other types of $K$, it is easy to verify Eq \eqref{NP-2} as well. This completes the proof of Theorem \ref{prop3.3}.

\end{document}